\documentclass[11pt]{amsart}

\usepackage{amsmath}
\usepackage{amssymb}
\usepackage{amsthm}
\usepackage{array}
\usepackage{fancyhdr}
\usepackage{euscript}
\usepackage{graphics,graphicx}
\usepackage{cancel}
\usepackage{fancybox}
\usepackage{verbatim}
\usepackage{tikz}

\usepackage{pstricks}
\usepackage{pst-plot}

\usepackage{setspace}
\newcommand{\ZZ}{\mathbf Z}

\newcommand{\QQ}{\mathbf Q}

\newtheoremstyle%
{custom}%
{}
{}
{}
{}
{}
{.}
{ }
{\thmname{}
\thmnumber{}%
\thmnote{\bfseries #3}}%

\newtheoremstyle%
{Theorem}%
{}%
{}%
{\itshape}%
{}%
{}%
{.}%
{ }%
{\thmname{\bfseries #1}%
\thmnumber{\;\bfseries #2}%
\thmnote{\;(\bfseries #3)}}%

\theoremstyle{Theorem}
\newtheorem{thm}{Theorem}[section]
\newtheorem{cor}[thm]{Corollary}
\newtheorem{lem}[thm]{Lemma}
\newtheorem*{mainthm}{Main Theorem}

\newtheorem*{maincor}{Corollary}
\theoremstyle{definition}
\newtheorem{dff}[thm]{Definition}
\newtheorem{xmp}[thm]{Example}
\newtheorem{rmk}[thm]{Remark}

\def\mbf{\mathbf}
\def\W{\mathbf{W}}

\def\tn{\textnormal}
\def\mcal{\mathcal}
\def\goth{\mathfrak}
\def\Map{\textnormal{Map}}

\title{A Witt--Burnside ring attached to a pro-dihedral group}
\author{Lance Edward Miller }
\address{ Department of Mathematics, University of Utah, Salt Lake City, UT, 84103}
\email{lmiller@math.utah.edu}
\thanks{The author was partially supported by a National Science Foundation VIGRE Grant, \# 0602219.}
\keywords{Witt vectors, Witt-Burnside rings, profinite, pro-dihedral}
\date{}

\begin{document}

\maketitle

\begin{abstract}
The ring of $p$-typical Witt vectors are an indispensable tool in number theory and mixed characteristic commutative algebra. Witt vectors were significantly generalized by Dress and Siebeneicher by producing for any profinite group G, a ring valued functor $\W_G$, The $p$-typical Witt vectors are recovered as the example $G = \ZZ_p$. This article explores the structure of the ring $W_{D_{2^\infty}}(k)$ where $k$ is a field of characteristic $2$ and $D_{2^\infty} := \varprojlim D_{2^n}$.
\end{abstract}

\section{Introduction}

In this article, we examine the algebraic structure of a Witt-Burnside ring $\W_G(k)$ where $G$ is a non-abelian pro-$2$ group and $k$ is a field of characteristic $2$. The Witt-Burnside functor attached to a profinite group $G$, introduced by Dress and Siebeneicher \cite{DS88}, is an endofunctor of the category of commutative rings that generalizes the usual Witt vector construction. We denote this functor by $\W_G$ and rings of the form $\W_G(R)$ where $R$ is a commutative ring are called Witt-Burnside rings. The $p$-typical Witt vectors are the Witt-Burnside rings attached to the additive group $G = \ZZ_p$ and Cartier's `big' Witt vectors are the Witt-Burnside rings attached to $\widehat{\ZZ}$. The name Burnside is used because the functor $\W_G$ applied to the integers is the burnside ring of the group $G$. 

The importance of $p$-typical Witt vectors in algebra and number theory cannot be understated and their applications are too numerous to give a complete account. We refer the reader to the extensive survey by Hazewinkle \cite{Haz78} on the applications of both $p$-typical and Cartier's `big' Witt vectors.

Further applications of the more general Witt-Burnside rings are difficult to determine without a careful understanding of the structure of the rings in question and their universal properties. J. Elliott gave a universal description \cite{Ell06} unifying different constructions of Witt-Burnside rings and computed the Frobenius and Verschibung maps. Work by the author as well as Y.T. Oh has begun to clarify the structure of these rings \cite{Mil,Oh07,Oh09}, but to date no detailed non-abelian examples appear in the literature. This article fills that gap. 

Like the $p$-typical Witt vectors, for any field $k$, the Witt-Burnside ring $\W_G(k)$ is always a local ring \cite[Cor. 3.21]{Mil}, and when $G$ is an infinite pro-$p$ group $\W_G(k)$ has characteristic $0$ when $k$ is a field of characteristic $p$ \cite[Cor 2, pg 115-116]{DS88}. Previous work by the author examined the algebraic structure of $\W_G(k)$ where $G \cong \ZZ_p^d$ for $d > 1$ and $k$ is a field of positive characteristic, but not necessarily perfect \cite{Mil}. This comparison showed that $\W_{\ZZ_p^d}(k)$ for $d > 1$ lacks many of the nice properties enjoyed by the $d=1$ case where one recovers the $p$-typical Witt vectors, i.e., a $p$-adically complete DVR. In particular, it is noted that unlike the $p$-typical Witt vectors, $\W_{\ZZ_p^d}(k)$ is not a noetherian ring and not even coherent\footnote{Coherent rings are those for which finitely generated ideals are finitely presented. See \cite{Gla89} for more details on coherent rings.} when $d = 2$ \cite[Cor. 5.21]{Mil}. In fact, even the unique maximal ideal fails to be finitely generated, though in the $d = 2$ case, $\W_{\ZZ_p^2}(k)$ is reduced local ring of characteristic $0$ \cite[Thm. 6.12]{Mil}, an algebraic situation similar to rings of functions. Whether or not $\W_{\ZZ_p^d}(k)$ is reduced   for $d > 2$ remains an open question. 

This lack of similarity between the general and $p$-typical cases adds to the mystery of the images of these functors and motivates the need for more examples. It is in this spirit that the author has studied the structure of the non-abelian pro-$2$ group $D_{2^\infty} = \left\{ \left( \begin{smallmatrix} \pm 1 & a \\ 0 & 1 \end{smallmatrix} \right) : a \in \ZZ_2 \right\} = \ZZ_2 \rtimes \{ \pm 1\}$ and $k$ a field of characteristic $2$. 
From general theory recorded in \cite{Mil}, $\W_{D_2^\infty}(k)$ is known to be a local ring, complete in the profinite topology. Our main theorem concerns finite generation. 

\begin{mainthm}
\label{thm:MainTheorem} \textnormal{[Thm. \ref{thm:D2inftynotnoth}, Cor. \ref{cor:notcoh}]}
For $k$ a field of characteristic $2$, the ring $\W_{D_{2^\infty}}(k)$ is not coherent. In particular, it is not noetherian.  
\end{mainthm}

We also examine the nilpotent elements of $\W_{D_{2^\infty}}(k)$ where $k$ has characteristic $2$, and in contrast to $\W_{\ZZ_p^2}(k)$ for $k$ of characteristic $p$, which is reduced \cite[Theorem 6.12]{Mil}, there are nilpotent elements in $\W_{D_{2^\infty}}(k)$. 

\begin{maincor} \textnormal{[Cor. \ref{cor:nonabelnonred}]}
For $k$ a field of characteristic $2$, the ring $\W_{D_{2^\infty}}(k)$ is not reduced. 
\end{maincor}


The rest of the paper is organized as follows. Section \ref{sec:WB} outlines the basic theory of Witt-Burnside rings and Section \ref{sec:D2} discusses the structure of the profinite group $D_{2^\infty}$. The main technical work on the structure is in Section \ref{sec:main}. 

The author would like to thank the reviewers for their careful and thorough review and their many good suggestions which greatly improved the accuracy and presentation of the article. He would also like to thank Keith Conrad for many helpful discussions. 

\section{A review of Witt-Burnside rings}\label{sec:WB}


This section is a brief review of the Witt vector construction, but the reader is referred to \cite{DS88, Ell06, Mil} for more details. Like $p$-typical Witt vectors, Witt-Burnside rings are constructed utilizing generalized Witt polynomials associated to a profinite group $G$. The index set of these generalized polynomials is the set of isomorphism classes of discrete transitive $G$-sets, called the {\it{frame}} of $G$ and is denoted $\mcal{F}(G)$. For example, $\mcal{F}(\ZZ_p) = \mbf{N}$. The frame $\mcal{F}(G)$ is a partially ordered set. For $T$ and $U$ in $\mcal{F}(G)$ one has $U \leq T$ when there is a $G$-map from $T$ to $U$. Denote the set of all maps from $T$ to $U$ as $\Map_G(T,U)$ and the number of maps $\# \Map_G(T,U)$ by $\varphi_T(U)$. Thus $\varphi_T(U) \neq 0$ if and only if $T \leq U$. The following facts about $\mcal{F}(G)$ are easy to verify. 

\begin{enumerate}
\item For $T$ and $U$ in $\mcal{F}(G)$ with $U \leq T$, $\# U$ divides $\# T$ and $\# T / \# U$ represents the size of any of the fibers of any element of $\Map_G(T,U)$. 
\item If $T$ has a normal stabilizer then $\varphi_T(U) = \# U$ for $U \leq T$. 
\item For each $T$ in $\mcal{F}(G)$, there are only finitely many $U$ in $\mcal{F}(G)$ with $U \leq T$. 
\end{enumerate}


For $T \in \mcal{F}(G)$, 
define the $T$-th {\it Witt polynomial} to be 
\begin{equation}\label{Wdef}
W_T(\{X_U\}_{U \in \mcal{F}(G)}) = \sum\limits_{U \leq T} \varphi_T(U) X_U^{\#T / \#U} = X_0^{\#T} + \ldots + \varphi_T(T) X_T,
\end{equation}
where $0$ denotes the trivial $G$-set $G/G$. (Trivially $\varphi_T(0) = 1$.) 
This is a finite sum since there are only finitely many 
$U \leq T$.

Applying this for $G = \mbf{Z}_p$ one recovers the usual $n$-th Witt polynomial, though on a different index set. In Section \ref{sec:D2}, we describe the frame of $D_{2^\infty}$ as well as the Witt polynomials. The next theorem is fundamental to the construction of $\W_G$; see \cite[Lem. 2.1]{Ell06} and \cite[Thm 2.1]{Mil} for more details. 

\begin{thm}\label{invertthm}
Let $A$ be a commutative ring in which no nonzero element 
is killed under multiplication by $\varphi_T(T)$ for any $T \in \mcal{F}(G)$. 
Then the function
$\prod_{T \in \mcal{F}(G)} A \rightarrow 
\prod_{T \in \mcal{F}(G)} A$ given by 
$\mbf{a} \mapsto (W_T(\mbf{a}))_{T \in \mcal{F}(G)}$ is injective. This function is bijective provided each $\varphi_T(T)$ is invertible in $A$. 
\end{thm}

Applying Theorem \ref{invertthm} to the ring 
$A = \QQ[\underline{X},\underline{Y}]$ where $\underline{X} = (X_T)$ and $\underline{Y} = (Y_T)$, and to the vectors 
$(W_T(\underline{X}) + W_T(\underline{Y}))_{T \in \mcal{F}(G)}$ and 
$(W_T(\underline{X})W_T(\underline{Y}))_{T \in \mcal{F}(G)}$, it tells us there are unique families of polynomials $\{S_T(\underline{X},\underline{Y})\}$ and 
$\{M_T(\underline{X},\underline{Y})\}$ in 
$\QQ[\underline{X},\underline{Y}]$ satisfying 
$$
W_T(\underline{X}) + W_T(\underline{Y}) = W_T(\underline{S}) \text{ for all } T \in \mathcal{F}(G)
$$
and 
$$
W_T(\underline{X})W_T(\underline{Y}) = W_T(\underline{M}) \text{ for all } T \in \mathcal{F}(G).
$$
More explicitly, this says 
\begin{equation}\label{STF}
\sum_{U \leq T} \varphi_T(U) X_U^{\# T / \#U} + \sum_{U \leq T} \varphi_T(U) Y_U^{\# T / \#U} = 
\sum_{U \leq T} \varphi_T(U) S_U^{\# T / \#U}
\end{equation}
and
\begin{equation}\label{MTF}
\left(\sum_{U \leq T} \varphi_T(U) X_U^{\# T / \#U}\right)\left(\sum_{U \leq 
T} \varphi_T(U) Y_U^{\# T / \#U}\right) = 
\sum_{U \leq T} \varphi_T(U) M_U^{\# T / \#U}
\end{equation}
for all $T$.  
The polynomials $S_T$ and $M_T$ each only 
depend on the variables $X_U$ and $Y_U$ for $U \leq T$. 

A significant theorem of Dress and Siebeneicher \cite[p.~107]{DS88}, which generalizes Witt's theorem, says that
the polynomials $S_T$ and $M_T$ have coefficients in $\ZZ$. 
We call the $S_T$'s and $M_T$'s the Witt addition and multiplication polynomials, respectively. 

\begin{xmp}
\label{xmp:MTST}
Taking $T = 0$, 
$$
S_0(\underline{X},\underline{Y}) = X_0 + Y_0, \ \ \ 
M_0(\underline{X},\underline{Y}) = X_0Y_0.
$$  
If $T \cong G/H$ where $H$ is a maximal proper open 
subgroup, so $\{U \colon U \leq T\} = \{0,T\}$, solving for $S_T$ and $M_T$ in 
(\ref{STF}) and (\ref{MTF}) yields
$$
S_T = X_T + Y_T + 
\frac{(X_0+Y_0)^{\#T} - X_0^{\#T} - Y_0^{\#T}}{\varphi_T(T)},$$
$$
M_T = X_0^{\#T}Y_T + X_TY_0^{\#T} + \varphi_T(T)X_TY_T. 
$$
Further addition and multiplication 
polynomials could be very complicated to write out explicitly, as is already apparent 
for the $p$-typical Witt vectors if you try to go past the first two polynomials. 
\end{xmp}


Since $S_T$ and $M_T$ have integral coefficients, they can be evaluated on any ring, including 
rings where the hypotheses of Theorem \ref{invertthm} break down, like a ring of characteristic $p$ 
when $G$ is a pro-$p$ group.

Let $G$ be a profinite group. 
For any commutative ring $A$, the {\it Witt--Burnside ring} $\mbf{W}_G(A)$ as a set is 
the product space $\prod_{T \in \mathcal{F}(G)} A$, with 
elements written as $\mbf{a} = (a_T)_{T \in \mathcal{F}(G)}$. 
The ring operations on $\mbf{W}_G(A)$ are 
defined by
$\mbf{a} + \mbf{b} = (S_T(\mbf{a},\mbf{b}))_{T \in \mathcal{F}(G)}
$
and
$
\mbf{a} \cdot \mbf{b} = (M_T(\mbf{a},\mbf{b}))_{T \in \mathcal{F}(G)}.
$ This construction is functorial: when $f \colon A \rightarrow B$ 
define $\W_G(f) \colon \W_G(A) \rightarrow \W_G(B)$ 
by applying $f$ to the coordinates: 
$
\W_G(f)(\mbf{a}) = (f(a_T))_{T \in \mcal{F}(G)} \in \W_G(B).
$
Because the polynomials $S_T$ and $M_T$ 
have integral coefficients, $\W_G(f)$ is a ring homomorphism 
and composition of ring homomorphisms is respected (in the same direction), so 
$\W_G$ is a covariant functor from commutative rings to commutative rings.

Also, like the $p$-typical setting, one can collect all the Witt polynomials together, to get a ring homomorphism 
$W : \W_G(A) \to \prod_{T \in \mcal{F}(G)} A$ which is $W_T$ in the $T$-th coordinate:
$$W(\mbf{a}) = (W_T(\mbf{a}))_{T \in \mathcal{F}(G)} = \left( \sum_{U \leq T} \varphi_T(U) a_U^{\# T/ \# U} \right)_{T \in \mcal{F}(G)}.$$
This homomorphism is called the {\it ghost map} and its coordinates 
$W_T(\mbf{a})$ are called the {\it ghost components} of $\mbf{a}$.  In some cases it 
is quite useless: if $G$ is pro-$p$ and $A$ has characteristic $p$ then 
$W(\mbf{a}) = (a_0^{\#T})_{T \in \mcal{F}(G)}$, whose dependence 
on $\mbf{a}$ only involves $a_0$.   
If $A$ fits the hypothesis of Theorem \ref{invertthm} then 
the ghost map is injective (i.e., the ghost components of $\mbf{a}$ determine $\mbf{a}$). The ghost map is essential in many of the calculations involving Witt vectors. In particular, to prove an algebraic identity, one first reformulates the desired identity in the ring of Witt vectors over a $\ZZ$-torsion free ring where ghost components determine a Witt vector. Then, using functoriality, one deduces the desired result from this reformulation.

\section{The frame of $D_{2^\infty}$}\label{sec:D2}



Throughout this section $G$ will denote $D_{2^\infty}= \left\{ \left( \begin{smallmatrix} \pm 1 & a \\ 0 & 1 \end{smallmatrix} \right) : a \in \ZZ_2 \right\},$ and $k$ will be a field of characteristic $2$ unless otherwise stated. The group $G$ is naturally isomorphic to $\ZZ_2 \rtimes \{ \pm 1\}$.  Our first goal is to describe the frame of $G$, and then the Witt polynomials. We also discuss a bit about the topology of $\W_{G}(k)$. 


Even though there are many subgroups of each index in $G$, a basic calculation shows only three conjugacy classes of open subgroups of each index $2^n \geq 1$ for $n \geq 1$ 
which we denote as follows. Set $r = \left( \begin{smallmatrix} 1 & 1 \\ 0 &1 \end{smallmatrix}\right)$ and $s = \left( \begin{smallmatrix} -1 & 0 \\ 0 & 1 \end{smallmatrix} \right)$. Representatives for the conjugacy classes of open subgroups are $H_n = \langle r^{2^{n-1}} \rangle$, $K_n = \langle s, r^{2^n} \rangle $ and $K_n' = \langle rs, r^{2^n} \rangle.$ In particular, $$H_n = \{ \left( \begin{smallmatrix} 1 & a \\ 0 & 1 \end{smallmatrix} \right) \colon a \equiv 0 \bmod 2^{n-1} \ZZ_2\},$$ $$K_n = \{ \left( \begin{smallmatrix} \pm 1 & a \\ 0 & 1 \end{smallmatrix} \right) \colon a \equiv 0 \bmod 2^n \ZZ_2 \} \text{ and }$$ $$K_n' = \{ \left( \begin{smallmatrix} 1 & a \\ 0 & 1 \end{smallmatrix} \right) \colon a \equiv 0 \bmod 2^{n-1} \ZZ_2)\} \cup \{ \left( \begin{smallmatrix} -1 & b \\ 0 & 1 \end{smallmatrix} \right) \colon b \equiv -1 \bmod 2^{n-1} \ZZ_2 \}.$$

Set $T_n = G/H_n$, $U_n = G/K_n$ and $U_n' = G/K_n'$. So $\# T_n = \# U_n = \# U_n' = 2^n$. In $\mcal{F}(G)$, the cover of $T_n$ is $T_{n+1}$, the covers of $U_n$ are $U_{n+1}$ and $T_{n+1}$ and the covers of $U_{n}'$ are $U_{n+1}'$ and $T_{n+1}$. The bottom portion of the Hasse diagram for the lattice $\mcal{F}(G)$ is pictured in Figure \ref{fig:FrameD2infty}. Clearly $T_n \cong G$ as groups for any $n \geq 2$ and since neither $K_n$ nor $K_n'$ is normal, $U_n$ and $U_n'$ are not groups.

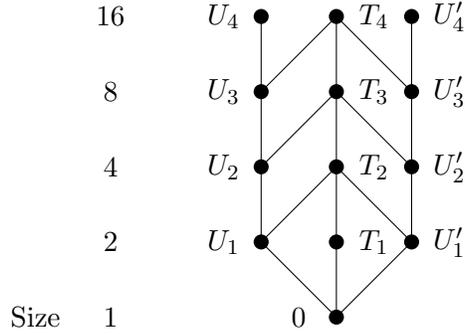
\begin{figure}[htp]
\centering
\begin{tikzpicture}[smooth]


\fill[black] (0,0) circle (0.1cm);

\fill[black] (-1,1) circle (0.1cm);
\fill[black] (0,1) circle (0.1cm);
\fill[black] (1,1) circle (0.1cm);

\fill[black] (-1,2) circle (0.1cm);
\fill[black] (0,2) circle (0.1cm);
\fill[black] (1,2) circle (0.1cm);

\fill[black] (-1,3) circle (0.1cm);
\fill[black] (0,3) circle (0.1cm);
\fill[black] (1,3) circle (0.1cm);

\fill[black] (-1,4) circle (0.1cm);
\fill[black] (0,4) circle (0.1cm);
\fill[black] (1,4) circle (0.1cm);

\draw[-] (0,0) -- (-1,1);
\draw[-] (0,0) -- (0,1);
\draw[-] (0,0) -- (1,1);

\draw[-] (-1,1) -- (-1,2);
\draw[-] (0,1) -- (0,2);
\draw[-] (1,1) -- (1,2);
\draw[-] (-1,1) -- (0,2);
\draw[-] (1,1) -- (0,2);

\draw[-] (-1,2) -- (-1,3);
\draw[-] (0,2) -- (0,3);
\draw[-] (1,2) -- (1,3);
\draw[-] (-1,2) -- (0,3);
\draw[-] (1,2) -- (0,3);

\draw[-] (-1,3) -- (-1,4);
\draw[-] (0,3) -- (0,4);
\draw[-] (1,3) -- (1,4);
\draw[-] (-1,3) -- (0,4);
\draw[-] (1,3) -- (0,4);

\node at (-0.5,0) {$0$} {};

\node at (-1.5,1) {$U_1$} {};
\node at (-1.5,2) {$U_2$} {};
\node at (-1.5,3) {$U_3$} {};
\node at (-1.5,4) {$U_4$} {};

\node at (0.5,1) {$T_1$} {};
\node at (0.5,2) {$T_2$} {};
\node at (0.5,3) {$T_3$} {};
\node at (0.5,4) {$T_4$} {};

\node at (1.5,1) {$U_1'$} {};
\node at (1.5,2) {$U_2'$} {};
\node at (1.5,3) {$U_3'$} {};
\node at (1.5,4) {$U_4'$} {};

\node at (-4,0) {Size} {};

\node at (-3,0) {$1$} {};
\node at (-3,1) {$2$} {};
\node at (-3,2) {$4$} {};
\node at (-3,3) {$8$} {};
\node at (-3,4) {$16$} {};

\end{tikzpicture}
\caption{Portion of the frame $\mcal{F}(G)$ where $G = D_{2^\infty}$.}
\label{fig:FrameD2infty}
\end{figure}


\begin{thm}
In $\W_{G}(\ZZ[\underline{X}])$:
\begin{eqnarray*}
W_{T_n}(\underline{X}) & = & \sum\limits_{V \leq T_n} \# V X_V^{\# T / \# V}  = X_0^{2^n} + \sum_{i=1}^n 2^i X_{T_i}^{2^{n-i}} +\sum_{i=1}^{n-1} 2^i (X_{U_i}^{2^{n-i}} +X_{U_i'}^{2^{n-i}}) \\
W_{U_n}(\underline{X}) & = & X_0^{2^n} + \sum\limits_{i=1}^n 2 X_{U_i}^{2^{n-i}} \\
W_{U_n'}(\underline{X}) & = & X_0^{2^n} + \sum\limits_{i=1}^n 2 X_{U_i'}^{2^{n-i}} 
\end{eqnarray*}
\end{thm}
\begin{proof}
Since each subgroup $H_n = \langle r^{2^{n-1}} \rangle$ is normal, 
for any $V \in \mcal{F}(G)$ and $V \leq T_n$, one has $\varphi_{T_n}(V) = \# V$. For any $m \leq n$, $U_m \leq U_n$. 
Recall $T_{n+1} = G/H_{n+1} \cong D_{2^{n}}$ which can be written $\langle r,s | r^{2^n}=1, s^2=1, srs^{-1} = r^{-1} \rangle$. 
Since $T_{n+1} > U_n \geq U_m$ we can view $U_m$ and $U_n$ in $\mcal{F}(D_{2^{n}})$ as the $D_{2^{n}}$-sets 
$D_{2^{n}}/\langle s,r^{2^m}\rangle$ and $D_{2^{n}}/\langle s, r^{2^n}\rangle$ respectively.  
As subsets of $D_{2^{n}}$, $\langle s, r^{2^n} \rangle = \langle s \rangle$ and $\# \langle s, r^{2^m} \rangle = 2^{n - m +1}$. By Section 2.1
the number of $D_{2^{n}}$-maps, $\varphi_{U_n}(U_m)$, is the same as
 $\# \{ g \in D_{2^{n}} : s \in g \langle s, r^{2^m} \rangle g^{-1} \}/ (2^{n - m +1})$. 
Since any conjugate of a rotation is another rotation and any conjugate of a reflection is another reflection, 
to test if $g$ satisfies $s \in g \langle s, r^{2^m} \rangle g^{-1}$ it suffices to check that $$s \in \{ gsg^{-1},gr^{2^m}sg^{-1}, gr^{2\cdot 2^m}sg^{-1}, gr^{3 \cdot 2^m} sg^{-1}, \ldots, gr^{(2^{n-m}-1)2^m}sg^{-1}\}.$$ 

Call this last set $S_g$. 
Let $g \in D_{2^{n}}$ and suppose $g = r^i$. In $D_{2^{n}}$, we have
$r^i r^js r^{-i} = r^{j + 2i}$, so $S_{r^i} =\{r^{2i}s,r^{2^m + 2i}s,r^{2 \cdot 2^{m} + 2i}, \ldots, r^{(2^{n-m}-1)2^m + 2i}s \}$. 
We will count the number of $i$ such that $s \in S_{r^i}$ as $i$ ranges in $\{0,1,2,\ldots,2^n-1\}$. 

Clearly if $i = 0$, $s \in S_{r^i}$. We also have $s \in S_{r^i}$ when $(2^{n-m}-1)2^m + 2i = 2^n$, 
which happens for $i = 2^{m-1}$. In fact, when $i \equiv 0 \bmod 2^{m-1}$ write $i = a2^{m-1}$ where $0 \leq a < 2^{n-(m-1)}$ and then $(2^{n-m}-a)2^m + 2a2^{m-1} = 2^n$ so $s \in S_{r^i}$.
Thus there are at least $2^{n - (m-1)} = 2^{n-m+1}$ values $i$ such that $s \in S_{r^i}$. Now suppose $i \equiv c \bmod 2^{m-1}$ with $c \in \{0,\ldots,2^{m-1}-1\}$. Write $i = c + d 2^{m-1}$ where $0 \leq d \leq 2^{n-m+1} - 1$ and suppose $s \in S_{r^i}$. 
For some $b \in \{0,1\,\ldots,2^{n-m} -1 \}$, $b2^m + 2i = b 2^m + 2c + d 2^m \equiv 0 \bmod 2^n$ so $c \equiv -(b+d)2^{m-1} \bmod 2^{n-1}$ 
and since $m \leq n$ and $0 \leq c < 2^{m-1}$ we get $c = 0$. So $b+d \equiv 0 \bmod 2^{n-m}$. For each $d$ there is a unique $b$ which makes this true. So $s \in S_{r^i}$ if and only if $i = d2^{m-1}$ and $0 \leq d \leq 2^{n-m+1} - 1$. So we have found there are exactly $2^{n-m+1}$ elements $g \in D_{2^{n}}$ such that $s \in S_g$ when $g = r^i$. A similar calculation for $g = r^is$ using the formula $(r^is)r^js(r^is)^{-1} = r^{2i-j}s$ 
shows there are exactly $2^{n-m+1}$ more elements $g \in D_{2^{n}}$ 
such that $s \in S_g$. Thus $\# \{ g \in D_{2^{n}} : s \in g \langle s, r^{2^m} \rangle g^{-1} \} = 2 \cdot 2^{n-m+1}$ and 
$\varphi_{U_n}(U_m) = 2$. A similar calculation shows $\varphi_{U_n'}(U_m') = 2$. 
\end{proof}

The ring $\W_G(k)$ comes with an inherent profinite topology and since it is local there is another topology defined by the unique maximal ideal $\goth{m} = \{ \mbf{a} \colon a_0 = 0 \}$. While in the $p$-typical Witt vectors, these two topologies agree, in general they can be different; for example they differ when $G = \ZZ_p^d$ for $d > 1$ \cite[Thm. 5.7]{Mil}. 
In between them there is a third topology, called the {\it initial vanishing } topology defined by the following filtration of ideals \cite[Def. 3.1]{Mil}. Set $$I_n(G,k) = \{ \mbf{a} \in \W_G(k) : a_T = 0 \mbox{ for } \# T < 2^n \}.$$  It is not hard to argue the these are in fact ideals and they obviously form a strictly decreasing filtration $I_1(G,k) \supset I_2(G,k) \supset \cdots$ \cite[Lem. 3.2]{Mil}. It is the case that $\W_G(k)$ is complete in this initial vanishing topology \cite[Thm. 3.5]{Mil} and in fact since $G$ is topologically finitely generated this topology agrees with the profinite topology \cite[Thm. 3.7]{Mil}. It is natural to ask whether or not this topology is defined by a graded sequence of ideals. In particular is it true that $I_m(G,k)I_n(G,k) \subset I_{m+n}(G,k)$? This holds provided $G$ satisfies a divisibility condition on $G$-sets called the ratio property \cite[Def. 3.14]{Mil}. For a general pro-$p$ group, this is the condition that when $T,T_1,T_2 \in \mcal{F}(G)$ such that $T \geq T_1$, $T \geq T_2$ and $\# T < \# T_1 \# T_2$, then $\varphi_T(T_1)\varphi_T(T_2)/\varphi_T(T)$ is an integral multiple of $p$. This is trivial satisfied if $G$ is abelian. We now show that $G = D_{2^\infty}$ also satisfies the ratio property. 

\begin{cor}
\label{cor:D2inftyratio}
The group $G = D_{2^\infty}$ satisfies the ratio property, i.e., for any $V,V_1,V_2 \in \mcal{F}(G)$ with $V \geq V_1$, $V \geq V_2$ and $\# V < \# V_1 \# V_2$, $$\frac{\varphi_V(V_1)\varphi_V(V_2)}{\varphi_V(V)}$$ is an integral multiple of $2$. 
\end{cor}
\begin{proof}
The $G$-sets $V_1$ and $V_2$ are both nontrivial, so $V$, $V_1$ and $V_2$ are each some $U_i$ or $T_i$. If $V = T_n$, then $\varphi_{V}(V_1) = \# V_1$, $\varphi_{V}(V_2) = \# V_2$ and $\varphi_V(V) = \# V$ so $\varphi_V(V_1)\varphi_V(V_2)/\varphi_V(V) = \# V_1 \# V_2 / \# V$ is an integral multiple of $2$ as all three factors are powers of $2$ and $\# V < \# V_1 \# V_2$. If $V = U_n$ then $V_1 = U_{m_1}$ and $V_2 = U_{m_2}$ for some integers $m_1$ and $m_2$ less than or equal to $n$. Since $\varphi_V(V_1) = \varphi_V(V_2) = \varphi_V(V) = 2$, again the ratio $\varphi_V(V_1)\varphi_V(V_2)/\varphi_V(V) = 2$ is an integral multiple of $2$. A similar argument works when $V = U_n'$. 
\end{proof}

In particular, $I_m(G,k)I_n(G,k) \subset I_{m+n}(G,k)$ for all $m,n$ therefore $\goth{m}^m = (I_1(G,k))^m \subset I_m(G,k)$ for all $m$. Corollary~\ref{cor:equalcoords} shows that this inclusion cannot be reversed. This also answers in the negative the following natural question. Is the topology generated by $\goth{m}$ is the same as this initial vanishing topology? This result is similar to the situation for $G = \ZZ_p^d$ for $d \geq 2$ where also the $\goth{m}$-adic and initial vanishing topologies differ. 

\section{The structure of $\W_{D_{2^\infty}}(k)$}\label{sec:main}

We now show the main results about failure of finite generation of ideals in $\W_{D^\infty}(k)$ with $k$ a field of characteristic $2$. First we establish some identities similar to 
\cite[Lem. 5.17, Lem. 5.18]{Mil}, whose proof goes through for any {\it abelian} pro-$p$ $G$. To modify, we utilize the notion of chain $G$-sets in $\mcal{F}(G)$.

\begin{dff} For any profinite group $G$, a $G$-set $T$ is called a {\it chain} $G$-set provided the collection partially order set $\{U \in \mcal{F}(G) \colon U \leq T\}$ forms a chain. 
\end{dff}


The proofs of \cite[Lem. 5.17, Lem. 5.18]{Mil}  utilized cyclic $\ZZ_p^d$-sets in $\mcal{F}(\ZZ_p^d)$; i.e., $\ZZ_p^d$-sets which are also cyclic groups.  For any profinite $G$, all cyclic $G$-sets are also chain $G$-sets. The chain $D_{2^\infty}$-sets are either $0$ or $U_n$ or $U_n'$. For the remainder of the section, $G$ will denote the group $D_{2^\infty}$ and $k$ is a field of characteristic $2$. 

\begin{lem}
\label{lem:D2notfg1}
Let $\mbf{b} \in \goth{m}$ and $\mbf{c} \in \W_G(k)$ and set $\mbf{m} = \mbf{b}\mbf{c}$. Let $U \in \mcal{F}(G)$ be a chain $G$-set. 
Then $m_U = b_Uc_0^{\# U}$. 
\end{lem}
\begin{proof}
Consider $R = \mbf{Z}[\underline{X},\underline{Y}]$ and let $\mbf{x},\mbf{y}$ be Witt vectors in $\W_G(R)$ with $x_V = X_V$ for $V \neq 0$, $x_0 = 0$, and $y_V = Y_V$ for all $V \in \mcal{F}(G)$. Let $\mbf{z} = \mbf{x}\mbf{y}$. We aim to show $z_U \equiv x_U y_0^{\# U} \bmod 2R$ as then functoriality yields the result. 

Note $z_0 = x_0y_0 = 0$, so we can assume $U \neq 0$. Since $U$ is chain, each $V \leq U$ is also. Assume by induction that for each $V < U$, $z_V \equiv x_V y_0^{\# V} \bmod 2R$. Now consider the equation $W_U(\mbf{z}) = W_U(\mbf{x})W_U(\mbf{y})$:\begin{equation}\label{eq:D2notfg11}\sum\limits_{0 < V \leq U} 2 z_V^{\# U / \# V} = \left(\sum\limits_{0 < V \leq U} 2 x_V^{\# U / \# V}\right)\left( y_0^{\#U} + \sum\limits_{0 < V \leq U} 2 y_V^{\# U / \# V}\right).\end{equation} Since $z_V \equiv x_V y_0^{\# V} \bmod 2R$ for $V < U$, $2z_V^{\# U / \# V} \equiv 2x_V^{\# U/ \#V } y_0^{\# U} \bmod 4R$ and so reducing (\ref{eq:D2notfg11}) mod $4R$ we get $$\sum_{0 < V < U} 2x_V^{\# U/ \#V } y_0^{\# U} + 2z_U \equiv \sum_{0 < V < U} 2x_V^{\# U/ \#V } y_0^{\# U} + 2x_U y_0^{\# U} \bmod 4R.$$ Removing like terms on either side and dividing by $2$ shows the desired congruence.  
\end{proof}

\begin{lem}
\label{lem:D2notfg2}
Let $\mbf{x},\mbf{y} \in \goth{m}$ and set $\mbf{s} = \mbf{x} + \mbf{y}$. Let $U \in \mcal{F}(G)$ be a chain $G$-set. Then $\mbf{s}_U = \mbf{x}_U + \mbf{y}_U$. 
\end{lem}
\begin{proof}
Let $R = \mbf{Z}[\underline{X},\underline{Y}]$ and let $\mbf{x},\mbf{y}$ be Witt vectors in $\W_G(R)$ with $x_0 = 0$ and $x_V = X_V$ for $V \neq 0$, and $y_0 = 0$ and $y_V = Y_V$ for $V \neq 0$. Set $\mbf{z} = \mbf{x} + \mbf{y}$. By functoriality, it suffices to show the congruence $z_U \equiv x_U + y_U \bmod 2R$ for any chain $G$-set $U$. 

Note $z_0 = x_0 +y_0 = 0$. So we can take $U \neq 0$. Since $U$ is chain, each $V \leq U$ is also. Assume by induction that for each $V < U$, $z_V \equiv x_V + y_V \bmod 2R$. Consider the equation $W_U(\mbf{z}) = W_U(\mbf{x}) + W_U(\mbf{y})$:\begin{equation}\label{eq:D2notfg21}\sum\limits_{0 < V \leq U} 2 z_V^{\# U / \# V} = \sum\limits_{0 < V \leq U} 2 \left(x_V^{\# U / \# V} +y_V^{\# U / \# V}\right).\end{equation} For $V < U$, since $z_V \equiv x_V + y_V \bmod 2R$, one has $z_V^{\# U / \# V} \equiv x_V^{\# U / \# V} + y_V^{\# U/ \# V} \bmod 2R$, and so dividing  (\ref{eq:D2notfg21}) by $2$ and reducing mod $2R$ we get $$\sum_{0 < V < U} \left( x_V^{\# U/ \#V } + y_V^{\# U / \# V} \right) + z_U \equiv \sum\limits_{0 < V < U} \left(x_V^{\# U / \# V} +y_V^{\# U / \# V}\right) + \left(x_U + y_U\right) \bmod 2R.$$ Removing like terms shows the desired congruence. 
\end{proof}

For a finite collection $\{\mbf{x}_1,\ldots,\mbf{x}_r\} \subset \W_G(k)$ we write $x_{i,T}$ for the $T$-th coordinate of $\mbf{x}_i$. 

\begin{cor}
\label{corgenD}
Let $r \geq 1$ be an integer and consider a finite set $\{\mbf{b}_1,\ldots,\mbf{b}_r\} \subset \goth{m}$ and Witt vectors $\{\mbf{c}_1,\ldots,\mbf{c}_r\} \subset \W_G(k)$. Set $\mbf{m} = \sum\limits_{i=1}^r \mbf{b}_i\mbf{c}_i$. For any chain $G$-set $U$, one has $m_U = \sum\limits_{i=1}^r b_{i,U}c_{i,0}^{\# U}$. 
\end{cor}
\begin{proof}
Use Lemma \ref{lem:D2notfg1} on each $\mbf{b}_i\mbf{c}_i$ and Lemma \ref{lem:D2notfg2} on their sum. 
\end{proof} 

Our first application of this is to show the initial vanishing topology differs from the $\goth{m}$-adic topology.  

\begin{cor}\label{cor:equalcoords}
For any natural number $n$, one has $I_n(G,k) \not\subseteq \goth{m}^n$. 
\end{cor}
\begin{proof}
From Corollary~\ref{corgenD}, any element  $\mbf{a} \in \goth{m}^n$ satisfies $a_T = 0$ for any chain $G$-set $T$. 
Since elements of $I_n(G,k)$ only must be zero in coordinates indexed by $G$-sets $T$ with $\# T < 2^n$ any element in $I_n(G,k)$ which is nonzero at $U_m$ with $m > n$ cannot lie in $\goth{m}^n$. 
\end{proof}

Next, we apply Corollary~\ref{corgenD} to show that $\goth{m}$ is not finitely generated. 

\begin{thm}
\label{thm:D2inftynotnoth}
The ideal $\goth{m}$ is not finitely generated. 
\end{thm}
\begin{proof}
Pick $r \geq 1$ and $\mbf{b}_1,\ldots,\mbf{b}_r \in \goth{m}$. Let $I$ be the ideal $(\mbf{b}_1,\ldots,\mbf{b}_r)$. We want to show $I \neq \goth{m}$. We proceed by contradiction. Suppose, for each $\mbf{a} \in I$ there are $\mbf{c}_1,\ldots,\mbf{c}_r \in \W_G(k)$ such that $\mbf{a} = \sum_{i=1}^r \mbf{b}_i\mbf{c}_i$. 

From Corollary \ref{corgenD}, $a_U = \sum\limits_{i=1}^r b_{i,U}c_{i,0}^{\# U}$ for each chain $G$-set $U$. For each integer $n$, collect these equations for $U \in \{U_1,U_2,\ldots,U_n\}$ into a matrix equation over $k$: 

\begin{displaymath} \left[ \begin{array}{c} a_{U_1} \\ a_{U_2} \\ \vdots \\ a_{U_n} \end{array} \right] = \left[ \begin{array}{cccccccccc} b_{1,U_1} & \ldots & b_{r,U_1} & 0 & \ldots & 0 & \ldots & 0 & \ldots & 0 \\
0 & \ldots & 0 & b_{1,U_2} & \ldots & b_{r,U_2} & \ldots & 0 & \ldots & 0 \\
\vdots &  & \vdots & \vdots &  & \vdots &  & \vdots &  & \vdots \\
0 & \ldots & 0 & 0 & \ldots & 0 & \ldots & b_{1,U_n} & \ldots & b_{r,U_n} \\
\end{array} \right] \left[ \begin{array}{c} c_{1,0}^2 \\ c_{2,0}^2 \\ \vdots \\ c_{r,0}^2 \\ c_{1,0}^4 \\ c_{2,0}^4 \\ \vdots \\  c_{r,0}^{4} \\ \vdots \\ c_{1,0}^{2^n} \\ c_{2,0}^{2^n} \\ \vdots \\ c_{r,0}^{2^n} \end{array}\right]. \end{displaymath} 

For each $\mbf{a} \in I = (\mbf{b}_1,\ldots,\mbf{b}_r)$, writing $\mbf{a}$ as a $\W_G(k)$-linear combination of $\mbf{b}_1,\ldots,\mbf{b}_r$ shows there are $c_{1,0},\ldots,c_{r,0}$ in $k$ which solve the matrix equation above. If $I = \goth{m}$ then we can choose $\mbf{a}$ so that $a_{U_i}$ and $a_{U_i'}$ are arbitrary. Thus the polynomial map $\varphi \colon k^r \to k^{n}$ given by 

\begin{displaymath} \left[ \begin{array}{c} c_{1,0} \\ c_{2,0} \\ \vdots \\ c_{r,0} \end{array} \right] \mapsto \left[ \begin{array}{cccccccccc} b_{1,U_1} & \ldots & b_{r,U_1} & 0 & \ldots & 0 & \ldots & 0 & \ldots & 0 \\
0 & \ldots & 0 & b_{1,U_2} & \ldots & b_{r,U_2} & \ldots & 0 & \ldots & 0 \\
\vdots &  & \vdots & \vdots &  & \vdots &  & \vdots &  & \vdots \\
0 & \ldots & 0 & 0 & \ldots & 0 & \ldots & b_{1,U_n} & \ldots & b_{r,U_n} \\
\end{array} \right] \left[ \begin{array}{c} c_{1,0}^2 \\ c_{2,0}^2 \\ \vdots \\ c_{r,0}^2 \\ c_{1,0}^4 \\ c_{2,0}^4 \\ \vdots \\  c_{r,0}^{4} \\ \vdots \\ c_{1,0}^{2^n} \\ c_{2,0}^{2^n} \\ \vdots \\ c_{r,0}^{2^n} \end{array}\right]. \end{displaymath}

\noindent would be surjective. If $k$ is finite then $\# (k^{r}) \geq \# (k^{n})$, so $r \geq n$. If $k$ is infinite then we have a surjective morphism $\varphi$ of varieties $k^r \to k^n$ which induces under pull back a $k$-algebra map $\varphi^* \colon k[X_1,\ldots,X_n] \to k[Y_1,\ldots,Y_r]$. The map $\varphi^*$ is injective as when $\varphi^*(g) = 0$ then we have $g(f_1,\ldots,f_n)  = 0$ in $k[Y_1,\ldots,Y_r]$ which means that $g(f_1(v),\ldots,f_n(v)) = 0$ for all $v \in k^r$. So $g(w) = 0$ for all $w \in k^n$ and so $g = 0$ in $k[X_1,\ldots,X_n]$. This gives an inclusion of function fields $k(X_1,\ldots,X_n) \hookrightarrow k(Y_1,\ldots,Y_r)$. Thus $\textnormal{tr deg}_k k(X_1,\ldots,X_n) \leq \textnormal{tr deg}_k k(Y_1,\ldots,Y_r)$ and so $n \leq r$. We started with no connection between $r$ and $n$, so by using $n > r$ we obtain a contradiction. 
\end{proof}

\begin{rmk}
This proof shows that for any finitely generated ideal $I = (\mbf{b}_1,\ldots,\mbf{b}_r)$ and $\mbf{a} \in I$, then the vector $(a_{U_1},\ldots,a_{U_n})$ cannot be an arbitrary element of $k^n$. This may be used as a general technique to restrict which ideals can be finitely generated. 
\end{rmk}

In addition to $\goth{m}$ not being finitely generated we can also realize $\goth{m}$ itself as an annihilator ideal of a particular Witt vector in $\W_G(k)$. In contrast, in $\W_{\ZZ_p^2}(k)$, the maximal ideal is not an annihilator ideal since this ring is local and reduced. This is sufficient to show that $\W_G(k)$ is not coherent as coherent rings must have finitely generated annihilator ideals \cite{Gla89}.

\begin{thm}
\label{thm:D2Ann}
Consider the Witt vector $\mbf{y} \in \W_G(k)$ such that 
\begin{displaymath}
y_{V} \equiv \begin{cases}
1, & \text{if $V = U_1$ or $V =  T_n$ for some $n \geq 0$,} \\
0, & \textrm{otherwise.}
\end{cases}
\end{displaymath}
Then one has $\tn{Ann}(\mbf{y}) = \goth{m}$. 
\end{thm}
\begin{proof}
Consider $R = \mbf{Z}[\underline{X}]$ and $\mbf{x} \in \W_G(R)$ such that $x_0 = 0$ and $x_V = X_V$ for $V \neq 0$. Consider the vector $\mbf{y}$ as above as an element of $\W_G(R)$ and set $\mbf{z} = \mbf{xy}$. We aim to show for each $V \in \mcal{F}(G)$, $z_V \equiv 0 \bmod 2R$. The result will then follow by functoriality. 
 
Clearly $z_0 = x_0y_0 = 0$. 
Consider the equation $W_{U_1}(\mbf{z}) = W_{U_1}(\mbf{x})W_{U_1}(\mbf{y})$. 
Expanding this out we have $z_0^2 + 2z_{U_1} = (x_0^2 + 2x_{U_1})(y_0 +2y_{U_1})$,
and since $z_0 = 0$, $x_0 = 0$, $y_0 = 0$, and $y_{U_1} = 1$, $2z_{U_1} = 4x_{U_1}$ so $z_{U_1} \equiv 0 \bmod 2R$ . 
Now assume by induction that $z_{U_i} \equiv 0 \bmod 2R$ for $1 \leq i \leq n-1$. 
Note that $W_{U_n}(\mbf{x})W_{U_n}(\mbf{y}) = \sum_{i = 1}^n 2 x_{U_i}^{2^{n-i}}\sum_{i = 1}^n 2 y_{U_i}^{2^{n-i}} = \sum_{i = 1}^n 4 x_{U_i}^{2^{n-i}}\equiv 0 \bmod 4R$. 
Expanding the equation $W_{U_n}(\mbf{z}) = W_{U_n}(\mbf{x}) W_{U_n}(\mbf{y})$ we have 
\begin{equation} \label{eq:NRD21}\sum_{i=1}^{n-1} 2 z_{U_i}^{2^{n-i}} + 2z_{U_n} \equiv 0 \bmod 4R.\end{equation} 
Since $z_{U_i} \equiv 0 \bmod 2R$ for $i = 1,2,\ldots,n-1$, $2 z_{U_i}^{2^{n-i}} \equiv 0 \bmod 4R$ and so we have $2 z_{U_n} \equiv 0 \bmod 4R$ or $z_{U_n} \equiv 0 \bmod 2R$. 

By definition of $\mbf{y}$, $W_{U_n'}(\mbf{y}) = 0$ so $W_{U_n'}(\mbf{z}) = W_{U_n'}(\mbf{x}) W_{U_n'}(\mbf{y}) = 0$ for all $n \geq 1$. Therefore by induction $z_{U_n'} = 0$ for all $n \geq 1$. 

It remains to show $z_{T_n} \equiv 0 \bmod 2R$ for all $n \geq 1$. Recall that $T_n$ is a normal $G$-set so $W_{T_n}(\underline{X}) =  \sum_{V \leq T_n} \# V X_V^{\# T_n / \# V}$. 
Note that $W_{T_1}(\mbf{x})W_{T_1}(\mbf{y}) = 4x_{T_1} \equiv 0 \bmod 4R$, so $W_{T_1}(\mbf{z}) = 2 z_{T_1} \equiv 0 \bmod 4R$. Thus $z_{T_1} \equiv  0 \bmod 2R$.

Now assume by induction that $z_{T_i} \equiv 0 \bmod 2R$ for $ i = 1,2,\ldots, n-1$ and consider the equation $W_{T_n}(\mbf{z}) = W_{T_n}(\mbf{x})W_{T_n}(\mbf{y})$. For all $V < T_n$,  $z_{V} \equiv 0 \bmod 2R$ so $\# V z_{V}^{2^n / \# V} \equiv 0 \bmod 2^{n+1}R$. Reducing $W_{T_n}(\mbf{z}) = W_{T_n}(\mbf{x})W_{T_n}(\mbf{y})$ modulo $2^{n+1}R$, we have $2^{n} z_{T_n} \equiv  W_{T_n}(\mbf{x})W_{T_n}(\mbf{y}) \bmod 2^{n+1}R$. We now expand \begin{equation}\label{eq:NRD22}W_{T_n}(\mbf{x})W_{T_n}(\mbf{y}) = \left(\sum_{V \leq T_n} \# V x_V^{\# T_n / \# V} \right)\left(\sum_{V \leq T_n} \# V y_V^{\# T_n / \# V} \right).\end{equation} Dropping zero terms from (\ref{eq:NRD22}) we have 
\begin{eqnarray*}
W_{T_n}(\mbf{x})W_{T_n}(\mbf{y}) & = & \left(\sum_{0 < V \leq T_n} \# V x_V^{\# T_n / \# V} \right)\left(2 y_{U_1}^{2^{n-1}} + \sum_{i = 1}^{n} 2^i y_{T_i}^{2^{n-i}}\right) \\
& = & \left(\sum_{0 < V \leq T_n} \# V x_V^{\# T_n / \# V} \right) \left(2 + \sum_{i=1}^{n} 2^i \right)\\
& = & \left(\sum_{0 < V \leq T_n} \# V x_V^{\# T_n / \# V} \right) \left(2^{n+1} \right),
\end{eqnarray*}
and so $2^{n} z_{T_n} \equiv 0 \bmod 2^{n+1}R$ which gives $z_{T_n} \equiv 0 \bmod 2R$. 
\end{proof}


\begin{cor}
\label{cor:nonabelnonred}
The ring $\W_G(k)$ is not reduced. 
\end{cor}
\begin{proof}
Since $\mbf{y} \in \goth{m}$, $\mbf{y}^2 = \mbf{0}$. 
\end{proof}

\begin{cor}\label{cor:notcoh}
The ring $\W_G(k)$ is not coherent. 
\end{cor}
\begin{proof}
Since coherent rings have finitely generated annihilator ideals \cite{Gla89}, it suffices to show some annihilator ideal is not finitely generated. By Theorem \ref{thm:D2Ann} the ideal $\goth{m}$ is an annihilator ideal and by Theorem \ref{thm:D2inftynotnoth} this ideal is not finitely generated. 
\end{proof}


\begin{thebibliography}{HY}
\bibitem[DS88]{DS88} A. Dress, and C. Siebeneicher, {\em The Burnside ring of profinite groups and the Witt vector construction}, Adv. Math., {\bf 70}, (1988), 87--132.
\bibitem[Ell06]{Ell06} J. Elliott, {\em Constructing Witt--Burnside rings}, Adv. Math., {\bf 203}, (2006),  319--363.
\bibitem[Gla89]{Gla89} S. Glaz, {\em{ Commutative Coherent Rings}}, Springer-Verlag, New York, 1989.
\bibitem[Haz78]{Haz78} M. Hazewinkel, {\em{Formal Groups and Applications}}, Academic Press, New York, 1978. 
\bibitem[Oh09]{Oh09} Y. Oh, {\em Decomposition of the Witt-Burnside ring and Burnside ring of an abelian profinite group}, Adv. Math., {\bf 222}, (2009), 485--526.
\bibitem[Oh07]{Oh07} Y. Oh, {\em $q$-deformation of Witt-Burnside rings}, Math. Z., {\bf 257}, (2007), 151--191.
\bibitem[Mil]{Mil} L. E. Miller, {\em On the Structure of Witt Vectors attached to pro-$p$ groups}, arXiv:1103.4644 submitted.
\end{thebibliography}
\end{document}